\documentclass[11pt]{article}

\usepackage{amsmath}
\usepackage{amsthm}
\usepackage{amssymb}
\usepackage{graphicx}
\usepackage{comment}
\usepackage{color}

\usepackage{latexsym,array,amsfonts,eucal}
\usepackage{mathrsfs,dsfont,bbm}
\usepackage{titlesec,fancyhdr,titletoc}
\usepackage{verbatim}

\newtheorem{theorem}{Theorem}[section]
\newtheorem{thm}{Theorem}[section]
\newtheorem{corollary}[thm]{Corollary}
\newtheorem{lemma}[thm]{Lemma}
\newtheorem{proposition}[thm]{Proposition}
\theoremstyle{definition}

\theoremstyle{remark}

\numberwithin{equation}{section}

\begin{document}

\title{Upper tail probabilities of integrated Brownian motions}
\author{Fuchang Gao\thanks{fuchang@uidaho.edu, Research partially supported by a grant from the Simons Foundation, \#246211}\\
Department of Mathematics, University of Idaho\\
83844 Moscow, USA
    \and
Xiangfeng Yang\thanks{xiangfeng.yang@liu.se}\\
Department of Mathematics, Link\"{o}ping University\\
SE-581 83 Link\"{o}ping, Sweden
}

\date{\today}% or replace '\today' by a specified date such as 'January 9, 2012'

\maketitle

\begin{abstract}
We obtain new upper tail probabilities of $m$-times integrated Brownian motions under the uniform norm and the $L^p$ norm. For the uniform norm, Talagrand's approach is used, while for the $L^p$ norm, Zolotare's approach together with suitable metric entropy and the associated small ball probabilities are used. This proposed method leads to an interesting and concrete connection between small ball probabilities and upper tail probabilities (large ball probabilities) for general Gaussian random variable in Banach spaces. As applications, explicit bounds are given for the largest eigenvalue of the covariance operator, and appropriate limiting behaviors of the Laplace transforms of $m$-times integrated Brownian motions are presented as well.
\end{abstract}

%how to define Keywords and MSC
\def\keywords{\vspace{.5em}\hspace{-2em}
{\textit{Keywords and phrases}:\,\relax%
}}
\def\endkeywords{\par}

\def\MSC{\vspace{.0em}\hspace{-2em}
{\textit{AMS 2010 subject classifications}:\,\relax%
}}
\def\endMSC{\par}

\keywords{Integrated Brownian motion, upper tail probability, small ball probability, metric entropy}

\MSC{60F10, 60G15}

\section{Introduction}\label{sec:introduction} Suppose that $m\geq0$ is an integer, and $\{W(t)\}_{t\geq0}$ is the standard Brownian motion starting at zero. The $m$-times integrated Brownian motions $\{X_m(t)\}_{t\geq0}$ are defined as $X_0(t)=W(t)$ and
\begin{align}\label{def:original-ibm}
X_m(t)=\int_0^tX_{m-1}(s)ds,\quad \text{ for }t\geq0 \text{ and }m\geq1.
\end{align}
From integrations by parts, it follows that $X_m$ in (\ref{def:original-ibm}) has a representation
\begin{align}\label{def:ibm}
X_m(t)=\frac{1}{m!}\int_0^t(t-s)^mdW(s),\quad \text{ for }t\geq0 \text{ and }m\geq0.
\end{align}
We use $A_m$ to denote the covariance operator of $X_m,$ namely,
$$A_mf(t)=\int_0^1K_m(s,t)f(s)ds$$
where $K_m(s,t)=\frac{1}{(m!)^2}\int_0^{\min\{s,t\}}(s-u)^m(t-u)^mdu$ is the covariance function of $X_m.$ Among various studies on $m$-times integrated Brownian motions (cf. \cite{Shepp-1966}, \cite{Khoshnevisan-Shi-1998}, \cite{Chen-Li-2003} and \cite{Gao-Hannig-Torcaso-2003}), we specially recall the results on small ball probabilities established in \cite{Chen-Li-2003} and \cite{Gao-Hannig-Torcaso-2003}. Namely, the exact asymptotics as $\epsilon\rightarrow0^+$ of $\log \mathbb{P}\left\{\sup_{0\leq t\leq1}|X_m(t)|\leq \epsilon\right\},$ $\log \mathbb{P}\left\{\|X_m\|_{L^p[0,1]}\leq \epsilon\right\}$ (with $1\leq p<\infty$) and $\mathbb{P}\left\{\|X_m\|_{L^2[0,1]}\leq \epsilon\right\}$ are achieved. It is then natural to investigate the rare events from the opposite side, that is, upper tail probabilities as $r\rightarrow\infty,$
\begin{align}\label{question}
\mathbb{P}\left\{\sup_{0\leq t\leq1}|X_m(t)|>r\right\} \text{ and }\mathbb{P}\left\{\|X_m\|_{L^p[0,1]}>r\right\}.
\end{align}

Based on the theory of Gaussian processes, it is quite easy to deduce exact asymptotics for $\log \mathbb{P}\left\{\sup_{0\leq t\leq1}|X_m(t)|>r\right\}$ and $\log\mathbb{P}\left\{\|X_m\|_{L^p[0,1]}>r\right\};$ see Section 8.3 in \cite{Lifshits-2012} and Section 3.1 in \cite{Ledoux-Talagrand-1991}. In this paper, we will firstly derive sharp asymptotics for $\mathbb{P}\left\{\sup_{0\leq t\leq1}|X_m(t)|>r\right\}$ and $\mathbb{P}\left\{\|X_m\|_{L^2[0,1]}>r\right\},$ which are summarized in the following theorem.

\begin{theorem}\label{th:precise}
(I). For $m=0,$
\begin{align}\label{m=0-uniform-norm}
\mathbb{P}\left\{\sup_{0\leq t\leq1}|X_0(t)|>r\right\}=\mathbb{P}\left\{\sup_{0\leq t\leq1}|W(t)|>r\right\}\sim \frac{4}{\sqrt{2\pi}}\cdot r^{-1}\cdot\exp\left\{-\frac{r^2}{2}\right\};
\end{align}
For $m\geq1,$
\begin{align}\label{m>0-uniform-norm}
\mathbb{P}\left\{\sup_{0\leq t\leq1}|X_m(t)|>r\right\}\sim \frac{2}{m!\sqrt{2\pi(2m+1)}}\cdot r^{-1}\cdot\exp\left\{-\frac{(m!)^2(2m+1)r^2}{2}\right\}.
\end{align}

(II). For $0<p<\infty$ and $m=0,$
\begin{align}\label{m=0-L-norm}
\mathbb{P}\left\{\|X_0\|_{L^p[0,1]}>r\right\}=\mathbb{P}\left\{\|W\|_{L^p[0,1]}>r\right\}\sim 2\sigma \pi^{-3/4}\left(\frac{\Gamma(\frac{1}{2}+\frac{1}{p})}{\Gamma(1+\frac{1}{p})}\right)^{1/2}\cdot r^{-1}\cdot\exp\left\{-\frac{r^2}{2\sigma^2}\right\}
\end{align}
where $\sigma=\left(\frac{2}{p\pi}\right)^{1/2}\left(1+\frac{p}{2}\right)^{(p-2)/(2p)}\frac{\Gamma(\frac{1}{2}+\frac{1}{p})}{\Gamma(1+\frac{1}{p})}$;

For $p=2$ and $m\geq1,$
\begin{align}\label{m>0-L-norm}
\mathbb{P}\left\{\|X_m\|_{L^2[0,1]}>r\right\}\sim c(\overrightarrow{\lambda^m})\cdot r^{-1}\cdot\exp\left\{-\frac{r^2}{2\lambda^m_1}\right\}
\end{align}
where $\overrightarrow{\lambda^m}=(\lambda^m_n)_{n\geq1}$ is the set of eigenvalues of the covariance operator $A_m$ of $X_m,$ $c(\overrightarrow{\lambda^m})$ is a constant depending on $\overrightarrow{\lambda^m},$ and $\lambda^m_1$ is the largest eigenvalue.
\end{theorem}

The cases $m=0$ in both (I) and (II) of Theorem \ref{th:precise} have been known for a while; see for instance Theorem 7.6 in \cite{Piterbarg-Fatalov-1995} and Theorem 1 in \cite{Fatalov-2003}. We thus will prove Theorem \ref{th:precise} only for $m\geq1.$ It is worthy to note that under the uniform norm the case $m=0$ and the case $m\geq1$ show different features: $\mathbb{P}\left\{\sup_{0\leq t\leq1}W(t)>r\right\}\sim2\mathbb{P}\left\{W(1)>r\right\}$ and $\mathbb{P}\left\{\sup_{0\leq t\leq1}X_m(t)>r\right\}\sim\mathbb{P}\left\{X_m(1)>r\right\}.$

As a simple application of Theorem \ref{th:precise}, we are able to give explicit bounds for the largest eigenvalue $\lambda^m_1$ of the covariance operator $A_m.$
\begin{corollary}\label{corollary}
For every $m\geq1,$ the largest eigenvalue $\lambda^m_1$ satisfies
\begin{align}\label{eigenvalue}
\frac{1}{(m+1)^2(2m+3)}\leq\lambda^m_1\cdot(m!)^2\leq\frac{1}{2m+1}.
\end{align}
\end{corollary}
In \cite{Gao-Hannig-Torcaso-2003}, estimates on $\lambda^m_n$ were given for large $n$ with a fixed $m.$ In \cite{Lifshits-Papageorgiou-Wozniakowski-2012}, estimates on $\lambda^m_1$ (and $\lambda^m_2$) were given for large $m.$ None of them are for a fixed $m$ and a fixed eigenvalue. But at the same time, estimates in (\ref{eigenvalue}) are worse than those in \cite{Lifshits-Papageorgiou-Wozniakowski-2012} when $m$ is large.
\begin{proof}[Proof of Corollary \ref{corollary}]
It is straightforward to check that for $1\leq p<\infty,$
$$\mathbb{P}\left\{\sup_{0\leq t\leq1}|X_{m+1}(t)|>r\right\}\leq\mathbb{P}\left\{\|X_m\|_{L^p[0,1]}>r\right\}\leq \mathbb{P}\left\{\sup_{0\leq t\leq1}|X_m(t)|>r\right\}.$$
We take $p=2,$ and use (\ref{m>0-uniform-norm}) and (\ref{m>0-L-norm}) to deduce that
$$(m!)^2(2m+1)\leq\frac{1}{\lambda^m_1}\leq\left((m+1)!\right)^2(2m+3)$$
which is equivalent to (\ref{eigenvalue}).
\end{proof}
The idea of the proof of Theorem \ref{th:precise} is simple. Under the uniform norm, we employ the method developed by Talagrand in \cite{Talagrand-1988}, while under the $L^2$ norm, an asymptotic is used regarding the $l^2$ norm which was derived by Zolotarev \cite{Zolotarev-1991} (see also \cite{Linde-1991} for generalizations). Unfortunately, for general $1\leq p<\infty,$ similar arguments will not work. The covariance operator $A_m: L^q[0,1]\rightarrow L^p[0,1]$ where $\frac{1}{p}+\frac{1}{q}=1,$ has a norm $$\|A_m\|_p:=\sup_{\|g\|_q\le 1}\|A_m g\|_p=\sup_{\|g\|_q\le1}\sup_{\|f\|_q\le 1}\int_0^1\int_0^1K_m(t,s)f(t)g(s)dtds.$$ If $p=2,$ then it is straightforward to see that $\|A_m\|_2=\lambda_1^m.$ Our second result works for general $1\leq p<\infty,$ but it is only an upper bound.

\begin{theorem}\label{th:upper-bound}
For $1\leq p<\infty$ and $m\geq1,$ the following upper bound holds
\begin{align}\label{upper-bound}
\mathbb{P}\left\{\|X_m\|_{L^p[0,1]}>r\right\}\leq c_1(m,p)\cdot \exp\left\{-\frac{r^2}{2\|A_m\|_p}+c_2(m,p)\cdot r^{\frac{2}{2m+3}}\right\}
\end{align}
where $c_1(m,p)$ and $c_2(m,p)$ are two positive constants depending on $m$ and $p.$
\end{theorem}
Note that the upper bound (\ref{upper-bound}) is not trivial. To see this, let us recall the Borell's inequality (cf. Section 2.1 in \cite{Adler-1990}). Suppose $\{Y(t)\}_{t\in T}$ is a centered Gaussian process with sample paths bounded a.s., where $T$ is some parametric set. Let $\|Y\|=\sup_{t\in T}Y(t)$ and $\sigma^2_T=\sup_{t\in T}\mathbb{E}(Y^2(t)).$ Then for $r>\mathbb{E}\|Y\|,$
\begin{align}\label{Borell}
\mathbb{P}\left\{\|Y\|>r\right\}\leq2\exp\left\{-\frac{(r-\mathbb{E}\|Y\|)^2}{2\sigma^2_T}\right\}.
\end{align}
Now we rewrite the $L^p$ norm as a uniform norm
\begin{align}\label{integral-representation}
\|X_m\|_{L^p[0,1]}=\sup_{g\in T}\int_0^1X_m(t)g(t)dt:=\sup_{g\in T}X_m(g)
\end{align}
with $T=\left\{g\in L^q[0,1]: \frac{1}{p}+\frac{1}{q}=1 \text{ and } \|g\|_{L^q[0,1]}\leq1\right\}.$ Then it follows from (\ref{Borell}) that
\begin{align}\label{temp}
\mathbb{P}\left\{\|X_m\|_{L^p[0,1]}>r\right\}\leq 2\exp\left\{-\frac{(r-\mathbb{E}\|X_m\|_{L^p[0,1]})^2}{2\|A_m\|_p}\right\}.
\end{align}
The leading term $\frac{r^2}{2\|A_m\|_p}$ coincides in (\ref{upper-bound}) and (\ref{temp}), but the next term $r^{\frac{2}{2m+3}}$ in (\ref{upper-bound}) is better than $r$ in (\ref{temp}). As an application of Theorem \ref{th:upper-bound}, we have the following estimates for the Laplace transforms of $m$-times integrated Brownian motions.
\begin{corollary}\label{corr-laplace}
For $m\geq1$ and $1\leq\theta<2,$ the following statements hold as $r\rightarrow\infty:$
\begin{align*}
\mathbb{E}\exp\left\{r\cdot\left(\sup_{0\leq t\leq 1}X_m(t)\right)^\theta\right\}&\sim \frac{1}{\sqrt{2-\theta}}\exp\left\{\frac{2-\theta}{2\theta}\left((m!)^2(2m+1)\right)^{\frac{\theta}{\theta-2}}(r\theta)^{\frac{2}{2-\theta}}\right\};\\
\mathbb{E}\exp\left\{r\cdot\left(\|X_m\|_{L^2[0,1]}\right)^\theta\right\}&\sim c(\overrightarrow{\lambda^m})\sqrt{\frac{2\pi}{(2-\theta)\lambda_1^m}}\exp\left\{\frac{2-\theta}{2\theta}\left(\lambda_1^m\right)^{\frac{\theta}{2-\theta}}(r\theta)^{\frac{2}{2-\theta}}\right\};\\
\mathbb{E}\exp\left\{r\cdot\left(\|X_m\|_{L^p[0,1]}\right)^\theta\right\}&\leq c_1(m,p,\theta)\exp\left\{c_2(m,p,\theta)r^{\frac{2}{(2-\theta)(2m+3)}}+c_3(m,p,\theta)r^{\frac{2}{2-\theta}}\right\},
\end{align*}
where $c_i, i=1,2,3,$ are three positive constants depending on $m, p$ and $\theta.$ In particular, the constant $c_3(m,p,\theta)=\frac{2-\theta}{2\theta}\theta^{2/(2-\theta)}(\|A_m\|_p)^{\theta/(2-\theta)}.$
\end{corollary}
The related results of Corollary \ref{corr-laplace} have been known for $m=0;$ see for instance \cite{Fatalov-2003-2} and references therein. The proof of Theorem \ref{th:upper-bound} is based on an upper bound estimate in \cite{Adler-1990} involving the metric entropy of $T$ endowed with the canonical metric, with the help of the small ball probabilities of $X_m.$ It turns out that such proposed method works far beyond $m$-times integrated Brownian motions. Our last result is to present an interesting and concrete connection between small ball probabilities and upper tail probabilities for general Gaussian random variables in Banach spaces.
\begin{theorem}\label{th:connection}
Let $X$ be a centered Gaussian random variable in Banach space $(E,\|\cdot\|)$ with dual space $(E^*,\|\cdot\|_*)$. Suppose ${\mathbb P}\left\{\|X\|\le \varepsilon\right\}\ge e^{-c_0\varepsilon^{-\alpha}|\log \varepsilon|^\beta}$ as $\varepsilon\to 0^+$, for some $c_0>0,$ $\alpha>0$ and $\beta\in {\mathbb R}$. Then
$${\mathbb P}\left\{\|X\|>\lambda\right\}\le c_1\exp\left\{-\frac1{2\sigma^2}\lambda^2+c_2\lambda^{\frac{\alpha}{\alpha+1}} (\log\lambda)^{\frac{\beta}{\alpha+1}}\right\},$$
where $\sigma^2=\sup_{\|g\|_*\le 1}{\mathbb E}|g(X)|^2$, and $c_1$ and $c_2$ are constants depending only on $c_0$, $\alpha$ and $\sigma$.
\end{theorem}
For $m$-times integrated Brownian motion, according to \cite{Chen-Li-2003}, the small ball probabilities of $X_m$ have the following form,
$$\lim_{\epsilon\rightarrow0}\epsilon^{\frac{2}{2m+1}}\log \mathbb{P}\left\{\|X_m\|_{L^p[0,1]}\leq\epsilon\right\}=-c(m,p),\quad 1\leq p<\infty,$$
for some positive constant $c(m,p)$ depending on $m$ and $p$. In this case we can take $\alpha=2/(2m+1)$ and $\beta=0$ in Theorem \ref{th:connection} which leads to
$$\mathbb{P}\left\{\|X_m\|_{L^p[0,1]}>r\right\}\leq c_1\exp\left\{c_2 r^{\frac{2}{2m+3}}\right\}\cdot \Phi(r/\sigma),$$
where
$$\sigma^2=\sup_{\|g\|_q\le 1}{\mathbb E}|g(X)|^2=\sup_{\|g\|_q\le 1}\int_0^1\int_0^1 K_m(t,s)g(t)g(s)dtdt,$$
and $q=p/(p-1)$. Note that $\sigma^2=\|A_m\|_p$. Indeed, it is trivial that $\sigma^2\le \|A_m\|_p$. To see the other direction, we notice that $K_m(t,s)$ is covariance kernel. Thus
\begin{align*}
\|A_m\|_p&=\sup_{\|g\|_q\le1}\sup_{\|f\|_q\le 1}\int_0^1\int_0^1K_m(t,s)f(t)g(s)dtds\\
&=\frac12 \sup_{\|g\|_q\le1}\sup_{\|f\|_q\le 1}\int_0^1\int_0^1K_m(t,s)[f(t)g(s)+f(s)g(t)]dtds\\
&=\sup_{\|g\|_q\le1}\sup_{\|f\|_q\le 1}\int_0^1\int_0^1K_m(t,s)\left[\frac{f(t)+g(t)}{2}\cdot \frac{f(s)+g(s)}{2}-\frac{f(t)-g(t)}{2}\cdot \frac{f(s)-g(s)}{2}\right]dtds\\
&\le \sup_{\|g\|_q\le1}\sup_{\|f\|_q\le 1}\int_0^1\int_0^1K_m(t,s)\left[\frac{f(t)+g(t)}{2}\cdot \frac{f(s)+g(s)}{2}\right]dtds\\
&\le \sigma^2,
\end{align*}
where the last inequality follows from the fact that $\|(f+g)/2\|_q\le 1$. This recovers Theorem \ref{th:upper-bound}.

\section{Proofs}\label{sec:proofs}
\subsection{Proof of Theorem \ref{th:precise}}\label{subsec:1}
As remarked before, we will only prove for $m\geq1.$ It is straightforward to deduce from (\ref{def:ibm}) that
$$\sup_{0\leq t\leq 1}\mathbb{E}(X_m(t))^2=\sup_{0\leq t\leq 1}\frac{1}{(m!)^2}\frac{t^{2m+1}}{2m+1}=\frac{1}{(m!)^2}\frac{1}{2m+1}$$
and the suprema occurs uniquely at $t=1.$ Then by the result in \cite{Talagrand-1988}, the asymptotic (\ref{m>0-uniform-norm}) is proved if the following holds
$$\lim_{h\rightarrow0}h^{-1}\mathbb{E}\sup_{t\in T_h}\left(X_m(t)-X_m(1)\right)=0$$
with $T_h=\left\{t\in [0,1]: \mathbb{E}(X_m(t)X_m(1))\geq \frac{1}{(m!)^2}\frac{1}{2m+1}-h^2\right\}.$ To see this, notice that for $t\in T_h,$
\begin{equation}\label{temp-sup}
\begin{aligned}
h^2&\geq\mathbb{E}X_m^2(1)-\mathbb{E}(X_m(t)X_m(1))\\
&=\frac{1}{(m!)^2}\left[\int_0^1(1-s)^{2m}ds-\int_0^t(t-s)^m(1-s)^mds\right]\\
&=\frac{1}{(m!)^2}\left[\int_t^1(1-s)^{2m}ds+\int_0^t(1-s)^m\left((1-s)^m-(t-s)^m\right)ds\right]\\
&\geq\frac{1}{(m!)^2}\left[\frac{(1-t)^{2m+1}}{2m+1}+\int_0^t(1-s)^m(1-t)(1-s)^{m-1}ds\right]\\
&=\frac{1}{(m!)^2}\left[\frac{(1-t)^{2m+1}}{2m+1}+(1-t)\left(\frac{1}{2m}-\frac{(1-t)^{2m}}{2m}\right)\right].
\end{aligned}
\end{equation}
For small $h,$ any $t\in T_h$ will be close to $1,$ we thus set such $t\in[1/2,1]$ in (\ref{temp-sup}). In this way,
\begin{equation}\label{temp-sup-1}
h^2\geq c(m)\cdot(1-t)
\end{equation}
for some positive constant $c(m)$ depending on $m.$ Therefore,
\begin{align*}
\lim_{h\rightarrow0}h^{-1}\mathbb{E}\sup_{t\in T_h}\left(X_m(t)-X_m(1)\right)&=\lim_{h\rightarrow0}h^{-1}\mathbb{E}\sup_{t\in T_h}-\int_t^1X_{m-1}(s)ds\\
&\leq\lim_{h\rightarrow0}h^{-1}\mathbb{E}\sup_{t\in T_h}\int_t^1|X_{m-1}(s)|ds\\
&\leq\lim_{h\rightarrow0}h^{-1}\mathbb{E}\int_{1-\frac{h^2}{c(m)}}^1|X_{m-1}(s)|ds
\end{align*}
where last inequality is from (\ref{temp-sup-1}). This limit is then obvious zero since $\sup_{0\leq s\leq 1}\mathbb{E}|X_{m-1}(s)|<\infty.$ We also notice that
\begin{align*}
2\mathbb{P}\left\{X_m(1)>r\right\}=\mathbb{P}\left\{|X_m(1)|>r\right\}&\leq\mathbb{P}\left\{\sup_{0\leq t\leq1}|X_m(t)|>r\right\}\\
&\leq 2\mathbb{P}\left\{\sup_{0\leq t\leq1}X_m(t)>r\right\}\sim2\mathbb{P}\left\{X_m(1)>r\right\}
\end{align*}
which proves (I).

For the proof of (II), we first recall the Karhunen-Lo\`{e}ve expansion for $X_m$ as follows
$$X_m(t)=\sum_{n=1}^\infty Z_n\sqrt{\lambda^m_n}f_n(t)$$
where $\{Z_n\}_{n\geq1}$ is a sequence of i.i.d. standard normal $N(0,1)$ random variables, $\{\lambda^m_n\}_{n\geq1}$ is the set of eigenvalues of the covariance operator $A_m,$ and $\{f_n(t)\}_{n\geq1}$ is the set of the associated eigenfunctions that forms an orthonormal basis of $L^2[0,1].$ Then we have the in law identity
$$\|X_m\|_{L^2[0,1]}=\left(\sum_{n=1}^\infty\lambda^m_n Z^2_n\right)^{1/2}.$$
Now the results in \cite{Zolotarev-1991} can be applied in such $l^2$ and
\begin{align*}
\mathbb{P}\left\{\|X_m\|_{L^2[0,1]}>r\right\}&=\mathbb{P}\left\{\left(\sum_{n=1}^\infty\lambda^m_n Z^2_n\right)^{1/2}>r\right\}\sim2\cdot \bar{c}(\overrightarrow{\lambda^m})\cdot\sqrt{\frac{\lambda^m_1}{2\pi}}\cdot r^{-1}\cdot\exp\left\{-\frac{r^2}{2\lambda^m_1}\right\}
\end{align*}
where $\bar{c}(\overrightarrow{\lambda^m})$ is a constant depending on the eigenvalues $\{\lambda^m_n\}_{n\geq1}$ whose exact expression is $\bar{c}(\overrightarrow{\lambda^m})=\prod_{n=2}^\infty\left(1-\lambda^m_n/\lambda^m_1\right)^{-1/2}.$ The fact that $0<\bar{c}(\overrightarrow{\lambda^m})<\infty$ can be seen as follows. Since $\lambda^m_1$ is the largest eigenvalue (with multiplicity $1;$ cf. \cite{Gao-Hannig-Torcaso-2003}), $1-\lambda^m_n/\lambda^m_1$ is always positive and less than $1.$ Therefore the convergence of the product is equivalent to the convergence of the series $\sum_{n=2}^\infty\lambda^m_n/\lambda^m_1.$ The convergence of eigenvalue series is a basic fact of a covariance operator.

\subsection{Proof of Theorem \ref{th:upper-bound}}
We recall that $T=\left\{g\in L^q[0,1]: \frac{1}{p}+\frac{1}{q}=1 \text{ and } \|g\|_{L^q[0,1]}\leq1\right\}$ and
\begin{align*}
\|X_m\|_{L^p[0,1]}=\sup_{g\in T}X_m(g)=\sup_{g\in T}\int_0^1X_m(t)g(t)dt.
\end{align*}
On the parametric set $T$ we define the canonical metric $d(f,g)=\sqrt{\mathbb{E}(X_m(f)-X_m(g))^2}.$ Let $N(\epsilon, T, d)$ be the minimum number of open balls of radius $\epsilon$ needed to cover $T,$ then $\log N(\epsilon, T, d)$ is the \textit{metric entropy} of $(T,d).$ The proof will make use of the following upper estimate of the metric entropy of $(T,d).$

\begin{lemma}\label{lemma-metric-entropy}
For some constant $c>0,$
$$\log N(\epsilon, T, d)\leq c\cdot \epsilon^{-\frac{1}{m+1}}.$$
\end{lemma}
\begin{proof}
We recall the Karhunen-Lo\`{e}ve expansion for $X_m$ (which was used in Section \ref{subsec:1}) as follows
$$X_m(t)=\sum_{n=1}^\infty Z_n\sqrt{\lambda^m_n}f_n(t).$$
There is an elegant connection between the small ball probability $\log \mathbb{P}\left\{\|X_m\|_{L^p[0,1]}\leq \epsilon\right\}$ and the metric entropy $\log N(\epsilon, S, \|\cdot\|_{l^2}),$ where
\begin{align}\label{s-l-2}
S=\left\{(c_1, c_2, \ldots)\in l^2: \quad c_n=\int_0^1g(t)\sqrt{\lambda^m_n}f_n(t)dt,\,\,\|g\|_{L^q[0,1]}\leq1\right\};
\end{align}
see \cite{Gao-Li-Wellner-2010} and \cite{Kuelbs-Li-1993}. We now show that $\log N(\epsilon, T, d)=\log N(\epsilon, S, \|\cdot\|_{l^2}).$ To this end, the covariance function $K_m(s,t)$ of $X_m$ can be written as
\begin{align*}
K_m(s,t)=\mathbb{E}\left(X_m(s)X_m(t)\right)=\mathbb{E}\left(\sum_{n=1}^\infty Z_n\sqrt{\lambda^m_n}f_n(s)\sum_{n=1}^\infty Z_n\sqrt{\lambda^m_n}f_n(t)\right)=\sum_{n=1}^\infty \lambda^m_nf_n(s)f_n(t).
\end{align*}
Therefore, the covariance operator
$$A_mg(t)=\int_0^1g(s)K_m(s,t)ds=\sum_{n=1}^\infty \lambda^m_nf_n(t)\int_0^1g(s)f_n(s)ds.$$
Thus the canonical metric
\begin{equation}\label{metric-identity}
\begin{aligned}
d^2(f,g)&=\mathbb{E}(X_m(f)-X_m(g))^2=\mathbb{E}\left(\int_0^1X_m(t)\left(f(t)-g(t)\right)dt\right)^2\\
&=\int_0^1\left(f(t)-g(t)\right)A_m\left(f(t)-g(t)\right)dt\\
&=\sum_{n=1}^\infty\lambda^m_n\left(\int_0^1\left(f(t)-g(t)\right)f_n(t)dt\right)^2\\
&=\sum_{n=1}^\infty c_n^2=\|\vec{c}\|_{l^2}^2,
\end{aligned}
\end{equation}
where $\vec{c}=(c_1, c_2, \ldots)$ with $c_n=\sqrt{\lambda^m_n}\int_0^1\left(f(t)-g(t)\right)f_n(t)dt.$ Now we can pair a point $g\in T$ and a point $\vec{c}\in S,$ then the identity (\ref{metric-identity}) implies that an $\epsilon$ ball of $g$ is also an $\epsilon$ ball of $\vec{c}.$ Thus $\log N(\epsilon, T, d)=\log N(\epsilon, S, \|\cdot\|_{l^2}).$

Now we find estimates on $\log N(\epsilon, S, \|\cdot\|_{l^2})$ with the help of small ball probabilities of $X_m.$ According to \cite{Chen-Li-2003}, the small ball probabilities of $X_m$ have the following form,
$$\lim_{\epsilon\rightarrow0}\epsilon^{\frac{2}{2m+1}}\log \mathbb{P}\left\{\|X_m\|_{L^p[0,1]}\leq\epsilon\right\}=-c(m,p),\quad 1\leq p<\infty,$$
for some positive constant $c(m,p)$ depending on $m$ and $p.$ From Proposition 3.1 in \cite{Gao-Li-Wellner-2010}, it follows
$$\log N(\epsilon, S, \|\cdot\|_{l^2})\leq c\cdot \epsilon^{-\frac{1}{m+1}}$$
for some positive constant $c.$ This completes the proof.
\end{proof}
We note that the same arguments yield $\log N(\epsilon, T, d)\geq c'\cdot \epsilon^{-\frac{1}{m+1}}$ with some constant $c'>0.$ Now we apply a result to estimate the upper tail probability by making use of metric entropy $\log N(\epsilon, T, d).$ More precisely, Theorem 5.4 in \cite{Adler-1990} says that if $\log N(\epsilon, T, d)\leq c\cdot \epsilon^{-\alpha},$ then
$$\mathbb{P}\left\{\|X_m\|_{L^p[0,1]}>r\right\}\leq c_1\exp\left\{c_2\cdot r^{\frac{2\alpha}{2+\alpha}}\right\}\cdot \Phi(r/\|A_m\|_p)$$
for two positive constants $c_1$ and $c_2,$ where $\Phi(r)=(2\pi)^{-1/2}\int_r^\infty e^{-x^2/2}dx.$ According to Lemma \ref{lemma-metric-entropy}, the parameter $\alpha=\frac{1}{m+1}.$ Then it is straightforward to derive (\ref{upper-bound}).

\subsection{Proof of Corollary \ref{corr-laplace}}
The proof is based on a result of \cite{Lifshits-1994} connecting the upper tail behavior of a supremum random variable and its Laplace transform. More precisely, let $\{\xi_t\}_{t\in T}$ be a bounded and centered Gaussian random function with an arbitrary parametric set $T,$ then Theorem 1 in \cite{Lifshits-1994} says, as $r\rightarrow\infty,$
\begin{equation}\label{lifshits}
\begin{aligned}
\mathbb{P}\left\{\sup_{t\in T}\xi_t>\left(r\theta\sigma_T^2\right)^{1/(2-\theta)}\right\}\sim&\sqrt{2-\theta}\,\,\mathbb{E}\exp\left\{r\cdot\left(\sup_{t\in T}\xi_t\right)^\theta1_{\left\{\sup_{t\in T}\xi_t>0\right\}}\right\}\\
&\cdot\exp\left\{-\left(r\theta \sigma_T^2\right)^{2/(2-\theta)}\frac{1}{\theta\sigma_T^2}\right\}\cdot\frac{\sigma_T}{\sqrt{2\pi}\left(r\theta \sigma_T^2\right)^{1/(2-\theta)}}
\end{aligned}
\end{equation}
where $\sigma_T^2=\sup_{t\in T}\mathbb{E}\xi(t)^2.$ Corollary \ref{corr-laplace} directly follows from (\ref{lifshits}) by taking $\xi=X_m$ with appropriate parametric sets $T.$ More specifically, $T=[0,1]$ yields the first asymptotics in Corollary \ref{corr-laplace}, and $T=\left\{g\in L^q[0,1]: \frac{1}{p}+\frac{1}{q}=1 \text{ and } \|g\|_{L^q[0,1]}\leq1\right\}$ yields the other asymptotics.

\subsection{Proof of Theorem \ref{th:connection}}
As used in the proof of Theorem \ref{th:upper-bound}, we need connections between small ball probabilities and metric entropy estimates which comes from the following facts.
\begin{proposition}\label{prop}
Let $X$ be a centered Gaussian random variable in Banach space $(E,\|\cdot\|)$ with dual space $(E^*,\|\cdot\|_*)$. Denote $B_{E^*}$ the closed unit ball of $E^*$, and for $g\in E^*$, define $\|g\|_X=\sqrt{{\mathbb E}|g(X)|^2}$. Then, for $\alpha>0$ and
$\beta\in {\mathbb R},$ there is a constant $c_1>0$ such that for all $0<\varepsilon<1,$
$$\log {\mathbb P}\{\|X\|<\varepsilon\}\le -c_1\varepsilon^{-\alpha}|\log\varepsilon|^\beta$$
if and only if there is a constant $c_2>0$ such that for all $0<\varepsilon<1,$
$$\log N(\varepsilon, B_{E^*}, \|\cdot\|_X)\ge
c_2\varepsilon^{-\frac{2\alpha}{2+\alpha}}|\log\varepsilon|^{\frac{2\beta}{2+\alpha}};$$
and for $\beta>0$ and $\gamma\in{\mathbb R},$ there is a constant $c_3>0$ such that for all $0<\varepsilon<1,$
$$\log {\mathbb P}\{\|X\|<\varepsilon\}\le -c_3|\log\varepsilon|^\beta (\log |\log \varepsilon|)^\gamma$$
if and only if there is a constant $c_4>0$ such that for all $0<\varepsilon<1,$
$$\log N(\varepsilon, B_{E^*}, \|\cdot\|_X)\ge c_4|\log \varepsilon|^\beta (\log |\log \varepsilon|)^\gamma .$$
Furthermore, the results also hold if the inequalities are reversed.
\end{proposition}
\begin{proof}
The result is a consequence of metric entropy duality and Kuelbs-Li connection between metric entropy and small ball probability. It can be
seen (in less explicit form) in \cite{Gao-2004}, and follows immediately from Proposition 3.1 in \cite{Gao-Li-Wellner-2010}. Indeed, without loss of
generality, we assume that $X=\sum_{i=1}^\infty f_i\xi_i$, where $f_i\in E$ and $\xi_i$ are i.i.d. $N(0,1)$ random variables. Then we have
$$\|X\|=\sup_{g\in B_{E^*}} \left|\sum_{i=1}^\infty g(f_i)\xi_i\right|.$$
Denote $T=\{(g(f_1),g(f_2),...): g\in B_{E^*}\}\subset l^2$. Then $T$ is symmetric and convex. It is straightforward to check that $N(\varepsilon, B_{E^*}, \|\cdot\|_X)=N(\varepsilon, T, \|\cdot\|_{2})$. Thus, the result follows immediately from Proposition 3.1 in \cite{Gao-Li-Wellner-2010}.
\end{proof}

\begin{proof}[Proof of Theorem \ref{th:connection}]
The result follows from combining Proposition~\ref{prop} above and the proof of Theorem 5.4 in \cite{Adler-1990}. Indeed,  if we denote $D(g, \varepsilon)=\{h\in B_{E^*}: \|h-g\|_X<\varepsilon\}$. Then, by Dudley's metric entropy bound, we have
\begin{align*}
{\mathbb E} \sup_{h\in D(g,\varepsilon)}h(X)&\le C\int_0^{\varepsilon} \sqrt{\log N(s, B_{E^*},\|\cdot\|_X)}ds.
\end{align*}
By the lower bound assumption on the small ball probability and using Proposition~\ref{prop}, we immediately obtain
$${\mathbb E} \sup_{h\in D(g,\varepsilon)}h(X)\lesssim \frac{C}{2}(\alpha+2) \varepsilon^{\frac{2}{2+\alpha}}|\log\varepsilon|^{\frac{\beta}{\alpha+2}}.$$
By Borell's inequality, we have
$${\mathbb P} \left\{\sup_{h\in D(g,\varepsilon) }h(X)>\lambda\right\}\le 2 \exp\left\{-\frac{\lambda^2}{2\sigma^2}+C(\alpha+2) \varepsilon^{\frac{2}{2+\alpha}}|\log\varepsilon|^{\frac{\beta}{\alpha+2}}\frac{\lambda}{2\sigma^2}\right\}.$$
Let $g_1, g_2, ..., g_m$ be an $\varepsilon$-net of $B_{E^*}$ under $\|\cdot\|_X$ distance with minimum cardinality. By Proposition~\ref{prop}, we have $m\le \exp\{C'\varepsilon^{-\frac{2\alpha}{2+\alpha}}|\log\varepsilon|^{\frac{2\beta}{2+\alpha}}\}$. Thus,
\begin{align*}
{\mathbb P}\{\|X\|>\lambda\}&={\mathbb P}\left\{\sup_{h\in B_{E^*}}h(X)>\lambda\right\}\\
&\le \sum_{i=1}^m{\mathbb P}\left\{\sup_{h\in D(g_i,\varepsilon) }h(X)>\lambda\right\}\\
&\le \exp\left\{C'\varepsilon^{-\frac{2\alpha}{2+\alpha}}|\log\varepsilon|^{\frac{2\beta}{2+\alpha}}\right\}\cdot  2 \exp\left\{-\frac{\lambda^2}{2\sigma^2}+C(\alpha+2) \varepsilon^{\frac{2}{2+\alpha}}|\log\varepsilon|^{\frac{\beta}{\alpha+2}}\frac{\lambda}{2\sigma^2}\right\}.
\end{align*}
The result follows by choosing $\varepsilon\sim c\lambda^{-\frac{\alpha+2}{2\alpha+2}}(\log \lambda)^{\frac{\beta}{2\alpha+2}}$.
\end{proof}

\textit{\textbf{Acknowledgment}}.
The second named author is grateful to M. Lifshits for stimulating discussions and useful suggestions.

\end{document}